\documentclass[english,12pt]{smfart}
\usepackage{amssymb}
\usepackage{amsfonts}
\usepackage{amscd}
\usepackage[T1]{fontenc}
\usepackage[all]{xypic}
\usepackage{mathrsfs}

\usepackage{color}

\usepackage{xcolor}

\newcommand\change[1]
{
{#1}}

\CompileMatrices

\swapnumbers

\def\deg{{\rm deg}}

\def\ac{{\overline{\rm ac}}}

\def\11{{\mathbf 1}}
\def\AA{{\mathbb A}}

\def\CC{{\mathbb C}}

\def\LL{{\mathbb L}}

\def\NN{{\mathbb N}}

\def\QQ{{\mathbb Q}}

\def\ZZ{{\mathbb Z}}

\def\cC{{\mathscr C}}

\def\cL{{\mathcal L}}

\def\cP{{\mathcal P}}
\def\cQ{{\mathcal Q}}

\def\cT{{\mathcal T}}

\newcommand{\Spec}{\ensuremath{\mathrm{Spec}\,}}

\newcommand{\Lie}{\ensuremath{\mathrm{Lie}}}

\newcommand{\ACF}{\ensuremath{\mathrm{ACF}}}
\newcommand{\redu}{\ensuremath{\mathrm{red}}}
\newcommand{\rig}{\ensuremath{\mathrm{rig}}}
\newcommand{\ord}{\ensuremath{\mathrm{ord}}}
\newcommand{\spl}{\ensuremath{\mathrm{sp}}}
\newcommand{\rb}{\ensuremath{\mathrm{b}}}
\newcommand{\SA}{\ensuremath{\mathrm{SAb}}}
\newcommand{\Var}{\mathrm{Var}}
\DeclareMathOperator*{\vdim}{dim}
\newcommand{\KVar}[1]{\ensuremath{K_0(\Var_{#1})}}
\newcommand{\Kmod}[1]{\ensuremath{K_0^{\mathrm{mod}}(\Var_{#1})}}
\newcommand{\KR}{\ensuremath{K^R_0(\Var_{k})}}
\newcommand{\MRloc}{\ensuremath{\mathcal{M}^R_{k}}}

\newcommand{\Mr}{\ensuremath{\mathcal{M}_k^+}}

\newcommand{\mX}{\ensuremath{\mathcal{X}}}

\newtheorem{theorem}[subsubsection]{Theorem}
\newtheorem{lem}[subsubsection]{Lemma}
\newtheorem{cor}[subsubsection]{Corollary}
\newtheorem{prop}[subsubsection]{Proposition}

\newtheorem{conjecture}[subsection]{Conjecture}
\newtheorem*{conjecture*}{Conjecture}

\theoremstyle{definition}
\newtheorem{definition}[subsubsection]{Definition}
\newtheorem{example}[subsubsection]{Example}

\newtheorem{def-prop}[subsubsection]{Proposition-Definition}
\newtheorem{def-theorem}[subsubsection]{Theorem-Definition}
\newtheorem{def-lem}[subsubsection]{Lemma-Definition}

\theoremstyle{remark}
\newtheorem{remark}[subsubsection]{Remark}

\newtheorem{question}[subsubsection]{Question}

\theoremstyle{plain}

\numberwithin{equation}{subsection}

\def\boxit#1#2{\setbox1=\hbox{\kern#1{#2}\kern#1}%
\dimen1=\ht1 \advance\dimen1 by #1
\dimen2=\dp1 \advance\dimen2 by #1
\setbox1=\hbox{\vrule height\dimen1 depth\dimen2\box1\vrule}%
\setbox1=\vbox{\hrule\box1\hrule}%
\advance\dimen1 by .4pt \ht1=\dimen1
\advance\dimen2 by .4pt \dp1=\dimen2 \box1\relax}

\renewcommand{\theequation}{\thesubsection.\arabic{equation}}

\begin{document}

\title[Chai's Conjecture and Fubini properties]
{Chai's Conjecture and Fubini properties of dimensional motivic
integration}

\author{Raf Cluckers}

\address{Universit\'e Lille 1, Laboratoire Painlev\'e, CNRS - UMR 8524, Cit\'e Scientifique, 59655
Villeneuve d'Ascq Cedex, France, and Katholieke Universiteit
Leuven, Department of Mathematics, Celestijnenlaan 200B, B-3001
Leu\-ven, Bel\-gium} \email{Raf.Cluckers@math.univ-lille1.fr}
\urladdr{http://math.univ-lille1.fr/$\sim$cluckers}

\author{Fran\c cois Loeser}
\address{Institut de Math\'ematiques de Jussieu,
UMR 7586 du CNRS,
Universit\'e Pierre et Marie Curie, Paris, France}
\email{Loeser@math.jussieu.fr}
\urladdr{http://www.math.jussieu.fr/$\sim$loeser/}


\author{Johannes Nicaise}
\address{Katholieke Universiteit Leuven, Department of Mathematics,
Celestijnenlaan 200B, B-3001 Leu\-ven, Bel\-gium}
\email{Johannes.Nicaise@wis.kuleuven.be}
\urladdr{http://wis.kuleuven.be/algebra/nicaise/}

\begin{abstract}
We prove that a conjecture of Chai on the additivity of the base
change conductor for semi-abelian varieties over a discretely
valued field is equivalent to a Fubini property for the
 dimensions of certain motivic integrals. We prove this Fubini
property when the valued field has characteristic zero.
\end{abstract}

\maketitle

\renewcommand{\partname}{}

\section{Introduction}
Let $R$ be a henselian discrete valuation ring with quotient field
$K$ and perfect residue field $k$. Let $G$ be a semi-abelian
variety over $K$, i.e., an extension of an abelian $K$-variety by
a $K$-torus. Then $G$ can be canonically extended to a smooth
separated commutative group scheme $\mathcal{G}$ over $R$, the
so-called N\'eron $lft$-model of $G$
 \cite[10.1.1]{neron}.  We say that $G$ has semi-abelian reduction if the
 identity component of the special fiber of $\mathcal{G}$ is a
 semi-abelian $k$-variety.

 In \cite{chai}, Chai introduced the {\em base change conductor}
 $c(G)$ of $G$, a positive rational number that measures the defect
 of semi-abelian reduction of $G$. Its precise definition is recalled in Definition \ref{def-chai}. The base change conductor
 vanishes if and only if $G$ has semi-abelian reduction. For algebraic tori, this invariant had previously been defined and
 studied by Chai and Yu \cite{chai-yu}. They proved the deep result that the base
 change conductor of a $K$-torus $T$ is invariant under isogeny. Applying an argument from \cite{GG}, they deduced that $c(T)$ equals one half of the Artin
 conductor of the cocharacter module of $T$. For semi-abelian varieties,
 however, no similar cohomological interpretation is known to hold in general; in fact, the base change conductor is not even invariant under isogeny \cite[\S6.10]{chai},
  and many of its properties remain mysterious. One of the main open
  questions is the following conjecture, formulated by Chai in
  \cite[\S8.1]{chai}.
\begin{conjecture}[Chai]\label{conj-chai}
Let $G$ be a semi-abelian $K$-variety which fits into an exact
sequence of algebraic $K$-groups
$$0\to T\to G\to A\to 0$$ with $T$ a $K$-torus and $A$ an
abelian $K$-variety.
 Then we have
$$c(G)=c(A)+c(T).$$
\end{conjecture}

 The fundamental difficulty underlying this conjecture is that an
 exact sequence of semi-abelian varieties does not give rise to an
 exact sequence of N\'eron $lft$-models, in general.
Chai proved the conjecture if $k$ is finite, using Fubini's
theorem for integrals with respect to the Haar measure on the
completion of $K$. He also proved the conjecture when $K$ has
mixed characteristic, using a different method and applying the
property that $c(T)$ only depends on the isogeny class of $T$. If
$k$ has characteristic zero (more generally, if $G$ obtains
semi-abelian reduction after a {\em tame} finite extension of
$K$), Chai's conjecture can be proven in an elementary way; see
\cite[4.23]{HaNi}.

 \medskip
 In the first part of the present paper, we show that, in arbitrary characteristic,
  Chai's conjecture is equivalent to a Fubini property for the dimensions of certain
{\em motivic} integrals (equation \eqref{eq-fubini} in Theorem
\ref{thm-fubini}). We then prove in the second part of the paper
that this Fubini property holds when $K$ has characteristic zero
(Theorem \ref{holds}). This yields a new proof of the conjecture
in that case, which is close in spirit to Chai's proof of the
finite residue field case.

The strength of
our approach lies in the fact that we combine \change{insights of} two theories of
motivic integration, namely, the geometric theory of motivic
integration on rigid varieties of Loeser and Sebag \cite{motrigid}
and the
model-theoretic approach of Cluckers and Loeser \cite{CL,mimix}.
Let us emphasize that the
Fubini property in \eqref{eq-fubini}  is not an \change{immediate corollary of the Fubini results} from \cite{mimix}; see Remark
\ref{rem-fubini}. We need to combine the theory in \cite{mimix}
with a new result (Theorem \ref{badim} and its corollary), which roughly states that
the virtual dimension of a motivic integral \change{over a fixed space} only depends on the
dimensions of the values of the integrand. This theorem may be of
independent interest.

 We hope that our reformulation
of Chai's conjecture in terms of motivic integrals will also shed
new light on the open case of the conjecture, when $k$ is
infinite, $K$ has positive characteristic and $G$ is wildly
ramified.

\medskip{\small During the preparation of this paper, the research of the authors
has received funding from the European Research Council under the European Community's Seventh Framework Programme (FP7/2007-2013) / ERC Grant Agreement nr. 246903 NMNAG and from the Fund for Scientific Research - Flanders (G.0415.10).}

\section{Preliminaries}
\subsection{Notation}
Throughout this article, $R$ denotes a Henselian discrete
valuation ring with quotient field $K$ and perfect residue field
$k$. We denote by $\mathfrak{m}$ the maximal ideal of $R$, by
$R^{sh}$ a strict henselization of $R$ and by $K^{sh}$ its field
of fractions. The residue field $k^s$ of $R^{sh}$ is an algebraic
closure of $k$. We denote by $\widehat{R}$ the $\mathfrak{m}$-adic
completion of $R$ and by $\widehat{K}$ its field of fractions.

For every ring $A$, we denote by $(\mathrm{Sch}/A)$ the category
of $A$-schemes. We consider the {\em special fiber functor}
$$(\cdot)_k:(\mathrm{Sch}/R)\to (\mathrm{Sch}/k):\mathcal{X}\mapsto
\mathcal{X}_k=X\times_R k$$ and the {\em generic fiber functor}
$$(\cdot)_K:(\mathrm{Sch}/R)\to (\mathrm{Sch}/K):\mathcal{X}\mapsto
\mathcal{X}_K=\mathcal{X}\times_R K.$$
 A variety over a ring $A$ is a reduced separated $A$-scheme of
finite type.


\subsection{N\'eron models and semi-abelian reduction}
A semi-abelian variety over a field $F$ is an extension of an
abelian $F$-variety by an algebraic $F$-torus. Let $G$ be a
semi-abelian variety over $K$. It follows from
\cite[10.2.2]{neron} that $G$ admits a N\'eron $lft$-model
$\mathcal{G}$ in the sense of \cite[10.1.1]{neron}. It is the
minimal extension of $G$ to a smooth separated group scheme over
$R$. We say that $G$ has semi-abelian reduction if the identity
component $\mathcal{G}^o_k$ of the special fiber of $\mathcal{G}$
is a semi-abelian $k$-variety. There always exists a finite
separable extension $L$ of $K$ such that $G\times_K L$ has
semi-abelian reduction. If $G$ is an abelian variety, then this is
Grothendieck's  Semi-Stable Reduction Theorem
\cite[IX.3.6]{sga7a}. If $G$ is a torus, then one can take for $L$
the splitting field of 
 $G$. The general
case is easily deduced from these special cases; see
\cite[3.11]{HaNi-compseries}.

Let $K'$ be a finite separable extension of $K$, and denote by
$R'$ the integral closure of $R$ in $K'$. We set $G'=G\times_K K'$
and we denote by $\mathcal{G}'$ the N\'eron $lft$-model of $G'$.
By the universal property of the N\'eron $lft$-model, there exists
a unique morphism of $R'$-schemes \begin{equation}\label{eq-h}
h:\mathcal{G}\times_R R'\rightarrow \mathcal{G}'
\end{equation}
that extends the natural isomorphism between the generic fibers.
 If $G$ has semi-abelian reduction, then $h$ is an open immersion \cite[3.1(e)]{sga7a}, which induces an
 isomorphism
$$(\mathcal{G}\times_R R')^o\rightarrow (\mathcal{G}')^o$$
between the identity components of $\mathcal{G}\times_R R'$ and
$\mathcal{G}'$ \cite[VI$_B$.3.11]{sga3.1}.
\subsection{The base change conductor}
Let $G$ be a semi-abelian variety over $K$.
 Let $K'$ be a finite separable extension of $K$ such that
$G'=G\times_K K'$ has semi-abelian reduction, and denote by
$e(K'/K)$ the ramification index of $K'$ over $K$. The morphism
\eqref{eq-h} induces an injective morphism
\begin{equation}\label{eq-Lie}
\Lie(h):\Lie(\mathcal{G})\otimes_R R'\rightarrow
\Lie(\mathcal{G}')\end{equation} of free $R'$-modules of rank
$\dim(G)$.

\begin{definition}[Chai \cite{chai}, Section 1]\label{def-chai}
The base change conductor of $G$ is defined by
$$c(G)=\frac{1}{e(K'/K)}\cdot \mathrm{length}_{R'}\left(\mathrm{coker}(\Lie(h))\right).$$
\end{definition}
This definition does not depend on the choice of $K'$. The base
change conductor is a positive rational number that vanishes if
and only if $G$ has semi-abelian reduction \cite[4.16]{HaNi}. One
can view $c(G)$ as a measure for the defect of semi-abelian
reduction of $G$.


\subsection{A generalization of Chai's conjecture}
In \cite[8.1]{chai}, Chai asks whether Conjecture \ref{conj-chai}
can be generalized as follows.

\begin{question}\label{ques-chai}
 Do we have
$c(G_2)=c(G_1)+c(G_3)$ for every exact sequence of semi-abelian
$K$-varieties $$0\to G_1\to G_2\to G_3\to 0?$$
\end{question}

If $G_1$, $G_2$ and $G_3$ are tori, this can be easily deduced
from the deep fact that the base change conductor of a torus is
one half of the Artin conductor of the cocharacter module
\cite{chai-yu}, in the following way.
\begin{prop}\label{prop-tori}
Let $$0\to G_1\to G_2\to G_3\to 0$$ be an exact sequence of
$K$-tori. Then $c(G_2)=c(G_1)+c(G_3)$.
\end{prop}
\begin{proof} The sequence of cocharacter modules
$$0\to X_\bullet(G_1)\to X_\bullet(G_2)\to X_\bullet(G_3)\to 0$$ is exact. Tensoring with
$\QQ$, we get a split exact sequence of
$\QQ[\mathrm{Gal}(L/K)]$-modules
$$0\to X_\bullet(G_1)\otimes_{\ZZ}\QQ\to X_\bullet(G_2)\otimes_{\ZZ}\QQ\to X_\bullet(G_3)\otimes_{\ZZ}\QQ\to 0$$
 where $L$
is the 
 splitting field of $G_2$. Thus the
Artin conductor of $X_{\bullet}(G_2)\otimes_{\ZZ}\QQ$ is the sum
of the Artin conductors of $X_{\bullet}(G_1)\otimes_{\ZZ}\QQ$ and
$X_{\bullet}(G_3)\otimes_{\ZZ}\QQ$. Since the base change
conductor of a torus is one half of the Artin conductor of the
cocharacter module \cite{chai-yu}, we find that
$c(G_2)=c(G_1)+c(G_3)$.
\end{proof}
\begin{cor}
If Conjecture \ref{conj-chai} holds, then Question \ref{ques-chai}
has a positive answer when $G_1$ is a torus.
\end{cor}
\begin{proof}
Assume that $G_1$ is a torus. For every semi-abelian $K$-variety
$G$, we denote by $G_{\mathrm{tor}}$ its maximal subtorus and by
$G_{\mathrm{ab}}=G/G_{\mathrm{tor}}$ its abelian part. We consider
the closed subgroup
$\widetilde{G}_2=(G_3)_{\mathrm{tor}}\times_{G_3}G_2$ of $G_2$.
 We have a short exact sequence of $K$-groups
\begin{equation}\label{eq-tori}
0\to G_1\to \widetilde{G}_2\to (G_3)_{\mathrm{tor}}\to
0\end{equation} so that $\widetilde{G}_2$ is an extension of
$K$-tori, and thus a torus. Moreover, the morphism
$$G_2/\widetilde{G}_2\to G_3/(G_3)_{\mathrm{tor}}=(G_3)_{\mathrm{ab}}$$ is an isomorphism,
so that $\widetilde{G}_2=(G_2)_{\mathrm{tor}}$ and
$(G_2)_{\mathrm{ab}}\cong (G_3)_{\mathrm{ab}}$. By Conjecture
\ref{conj-chai}, we have
$c(G_i)=c((G_i)_{\mathrm{tor}})+c((G_i)_{\mathrm{ab}})$ for
$i=2,3$. Applying Proposition \ref{prop-tori} to the sequence
\eqref{eq-tori}, we find that
$c((G_2)_{\mathrm{tor}})=c(G_1)+c((G_3)_{\mathrm{tor}})$. It
follows that $c(G_2)=c(G_1)+c(G_3)$.
\end{proof}
Below, we will follow a different approach. We will use the
invariance of the base change conductor of a torus under isogeny
to reduce Question \ref{ques-chai} to the case where the maximal
split subtorus $(G_3)_{\spl}$ of $G_3$ is trivial (of course, this
is always the case if $G_3$ is an abelian variety as in Conjecture
\ref{conj-chai}). Then we prove that, if $G_1$ is a torus and
$(G_3)_{\spl}$ is trivial, the additivity property of the base
change conductor in Question \ref{ques-chai} is equivalent to a
certain Fubini property for motivic integrals. We prove this
Fubini property when $K$ has characteristic zero. These arguments
do not use the invariance of the base change conductor of a torus
under isogeny.

\section{Motivic Haar measures on semi-abelian varieties}
%

\subsection{The Grothendieck ring of
varieties}\label{subsec-Groth} Let $F$ be a field.  We denote by
$\KVar{F}$ the {\it Grothendieck ring of varieties over} $F$. As
an abelian group, $\KVar{F}$ is defined by the following
presentation:
\begin{itemize}
\item {\em generators:} isomorphism classes $[X]$ of separated
$F$-schemes of finite type $X$,
 \item {\em relations:} if $X$ is
a separated $F$-scheme of finite type and $Y$ is a closed
subscheme of $X$, then
$$[X]=[Y]+[X\setminus Y].$$
These relations are called {\em scissor relations}.
\end{itemize}
By the scissor relations, one has $[X]=[X_{\redu}]$  for every
separated $F$-scheme of finite type $X$, where $X_{\redu}$ denotes
the maximal reduced closed subscheme of $X$. We endow the group
$\KVar{F}$ with the unique ring structure such that
$$[X]\cdot [X']=[X\times_F X']$$ for all separated $F$-schemes of finite type $X$
and $X'$. The identity element for the multiplication is the class
$[\Spec F]$ of the point. To any constructible subset $C$ of a
separated $F$-scheme of finite type $X$, one can associate an
element $[C]$ in $\KVar{F}$ by choosing a finite partition of $C$
into subvarieties $C_1,\ldots,C_r$ of $X$ and setting
$[C]=[C_1]+\ldots+[C_r]$. The scissor relations imply that this
definition does not depend on the choice of the partition. For a
detailed survey on the Grothendieck ring of varieties, we refer to
\cite{NiSe-K0}.

We denote by $\Kmod{F}$ the {\em modified Grothendieck ring of
varieties over} $F$ \cite[\S3.8]{NiSe-K0}. This is the quotient of
$\KVar{F}$ by the ideal $\mathcal{I}_F$ generated by elements of
the form $[X]-[Y]$ where $X$ and $Y$ are separated $F$-schemes of
finite type such that there exists a finite, surjective, purely
inseparable $F$-morphism $Y\to X.$ If $F$ has characteristic zero,
then it is easily seen that $\mathcal{I}_F$ is the zero ideal
\cite[3.11]{NiSe-K0}, so that $\KVar{F}=\Kmod{F}$. It is not known
if $\mathcal{I}_F$ is non-zero if $F$ has positive characteristic.
In particular, if $F'$ is a non-trivial finite purely inseparable
extension of $F$, it is not known whether $[\Spec F']\neq 1$ in
$\KVar{F}$.

There exists a canonical isomorphism from $\Kmod{F}$ to the
Grothendieck ring $K_0(\ACF_F)$ of the theory $\ACF_F$ of
algebraically closed fields over $F$ \cite[3.13]{NiSe-K0}. One may
also consider the semi-ring variant $K^+_0(\ACF_F)$ of the ring
$K_0(\ACF_F)$, defined as follows. Let $\cL_{\rm ring}(F)$ be the
ring language with coefficients from $F$. As a semi-group,
$K^+_0(\ACF_F)$ is the quotient of the free commutative semi-group
generated by a symbol $[X]$ for each $\cL_{\rm ring}(F)$-definable
set, with zero-element $[\emptyset]$, and divided out by the
relations
\begin{itemize}
 \item if $X$ and $Y$ are  $\cL_{\rm ring}(F)$-definable subsets of a common $\cL_{\rm ring}(F)$-definable set, then
$$
[X\cup Y] + [X\cap Y] =[X] + [Y];
$$
\item if there exists an $\cL_{\rm ring}(F)$-definable bijection $X\to Y$ for the theory $\ACF_F$,
 then $[X]=[Y]$.
\end{itemize}
The semi-group $K^+_0(\ACF_F)$ carries a structure of semi-ring, induced by taking Cartesian products, $[X][Y]=[X\times Y]$.


If $R$ has equal characteristic, then we put
$$\KR=\KVar{k}.$$
If $R$ has mixed characteristic, then we put
$$\KR=\Kmod{k}.$$
We denote by $\LL$ the class of the affine line $\AA^1_k$ in $\KR$
and also in $K^+_0(\ACF_k)$. We will write $\MRloc$ for the
localization of $\KR$ with respect to $\LL$, and $\Mr$ for the
localization of $K^+_0(\ACF_k)$ with respect to $\LL$ and the
elements $\LL^i-1$ for all $i>0$.

For every element $\alpha$ of $\KR$, we denote by $P(\alpha)$ its
{\em Poincar\'e polynomial} \cite[4.13]{NiSe-K0}. This is an
element of $\ZZ[T]$, and the map
$$P:\KR\to \ZZ[T]:\alpha\mapsto P(\alpha)$$ is a ring morphism.
Hence, the map
$$
P^+:K^+_0(\ACF_k) \to \ZZ[T],
$$
obtained by composing $P$ with the canonical morphism
$K_0^+(\ACF_F)\to K_0(\ACF_F)\cong \KR$, is a morphism of
semi-rings. When $\alpha$ is the class of a separated $k$-scheme
of finite type $X$, then for every $i\in \NN$, the coefficient of
$T^i$ in $P(\alpha)$ is $(-1)^i$ times the $i$-th {\em virtual
Betti number} of $X$. The degree of $P(\alpha)$ is twice the
dimension of $X$ \cite[8.7]{Ni-tracevar}.

We have $P(\LL)=T^2$, so that $P$ localizes trough a ring morphism
$$P:
\MRloc \to \ZZ[T,T^{-1}]$$ and $P^+$ localizes through a semi-ring
morphism
$$P^+:
\Mr \to \ZZ[T,T^{-1},(T^{2i}-1)^{-1}]_{i>0}.
$$

\begin{definition}\label{def-dim}
Let $\alpha$ be an element of $\MRloc$, resp.~of $\Mr$. We define
the {\em virtual dimension} of $\alpha$ as $1/2$ times the degree
of the Poincar\'e polynomial $P(\alpha)$, resp.~$P^+(\alpha)$,
with the convention that the degree of the zero polynomial is
$-\infty$ and $(1/2)\cdot (-\infty)=-\infty$. We denote the
virtual dimension of $\alpha$ by $\vdim(\alpha)$.
\end{definition}

By definition, the virtual dimension is an element of $(1/2)\cdot
\ZZ\cup \{-\infty\}$. For every separated $k$-scheme of finite
type $X$ and every integer $i$, we have
$$\vdim([X]\LL^{i})=\dim(X)+i.$$

\subsection{Motivic integration on $K$-varieties}\label{subsec-motint}
Let $X$ be a $K$-variety. We say that $X$ is {\em bounded} if
 $X(K^{sh})$ is bounded in $X$ in the
sense of \cite[1.1.2]{neron}. If $X$ is a smooth $K$-variety, then
by \cite[3.4.2 and 3.5.7]{neron}, $X$ is bounded if and only if
$X$ admits a {\em weak N\'eron model} $\mX$. This means that $\mX$
is a smooth $R$-variety endowed with an isomorphism $\mX_K \to X$
such that the natural map
$$\mX(R^{sh}) \to X(K^{sh})$$ is a bijection.

The theory of motivic integration on rigid varieties was developed
in \cite{motrigid}, and further extended in \cite{NiSe-weilres}
and \cite{Ni-trace}. We refer to \cite{NiSe-survey} for a survey;
see in particular \cite[\S2.4]{NiSe-survey} for an erratum to the
previous papers. One of the main results can be reformulated for
algebraic varieties as follows.
 Let $X$ be bounded smooth
 $K$-variety of pure dimension, and let $\omega$ be a
gauge form on $X$, i.e., a nowhere vanishing differential form of
degree $\dim(X)$. Let $\mathcal{X}$ be a weak N\'eron model for
$X$. For every connected component $C$ of $\mX_k=\mX\times_R
 k$, we denote by $\ord_C\omega$ the order of $\omega$ along $C$.
 If $\varpi$ is a uniformizer in $R$, then $\ord_C\omega$ is the
 unique integer $n$ such that $\varpi^{-n}\omega$ extends to a
 generator of
 $\Omega^{\dim(X)}_{\mX/R}$ at the generic point of $C$. 
\begin{def-theorem}
 The object
 \begin{equation}\label{eq-motint}
 \int_X|\omega|=\LL^{-\dim(X)}\sum_{C\in
 \pi_0(\mX_k)}[C]\LL^{-\ord_C \omega}\quad \in
 \MRloc\end{equation}
 only depends on $X$ and $\omega$, and not on the choice of a weak
 N\'eron model $\mathcal{X}$. We call it the motivic integral of
 $\omega$ on $X$.
 \end{def-theorem}
 \begin{proof}
By \cite[4.9]{Ni-tracevar}, the formal $\frak{m}$-adic completion
of $\mathcal{X}$ is a formal weak N\'eron model of the rigid
analytification $X^{\rig}$ of $X\times_K \widehat{K}$, so that the
result follows from \cite[2.3]{HaNi}.
 \end{proof}
 It is clear from the definition that the motivic integral of
 $\omega$ on $X$ remains invariant if we multiply $\omega$ with a
 unit in $R$.

 \begin{remark}
 In the literature, the factor $\LL^{-\dim(X)}$ in the right hand side of \eqref{eq-motint} is sometimes  omitted (for instance in \cite{NiSe-survey}); this depends on the
 choice of a normalization for the motivic measure.
\end{remark}

\subsection{Motivic Haar measures}\label{subsec-mothaar}
Consider a semi-abelian $K$-variety $G$ of dimension $g$. We
denote by $\mathcal{G}$ the N\'eron $lft$-model of $G$ and by
$\Omega_G$ the free rank one $R$-module of translation-invariant
differential forms in $\Omega^{g}_{\mathcal{G}/R}(\mathcal{G})$.
 Note that $\Omega_G\otimes_R K$ is
canonically isomorphic to the $K$-vector space of
translation-invariant differential forms of maximal degree on $G$,
so that we can view $\Omega_G$ as an $R$-lattice in this vector
space. We denote by $\omega_G$ a generator of $\Omega_G$. It is
unique up to multiplication with a unit in $R$.

Let $K'$ be a finite separable extension of $K$ such that
$G'=G\times_K K'$ has semi-abelian reduction, and let $d$ be the
ramification index of $K'$ over $K$. Denote by $R'$ the
normalization of $R$ in $K'$. Dualizing the morphism
\eqref{eq-Lie} and taking determinants, we find a morphism of free
rank one $R'$-modules
$$\det(\Lie(h))^{\vee}:\Omega_{G'}\to \Omega_G\otimes_R R'$$ that
induces an isomorphism
$$ \Omega_{G'}\otimes_{R'} K'\cong \Omega_{G}\otimes_R K'
$$
 by tensoring
with $K'$.
 Thus we can view $\Omega_G$ as a sub-$R$-module of
 $\Omega_{G'}\otimes_{R'}K'$. This yields the following alternative
 description of the base change conductor.

 \begin{prop}\label{prop-bcdiff}
Let $\varpi'$ be a uniformizer in $R'$. The base change conductor
$c(G)$ of $G$ is the unique element $r$ of $\ZZ[1/d]$ such that
$$(\varpi')^{rd} \omega_G$$ generates the $R'$-module $\Omega_{G'}$.
 \end{prop}
 \begin{proof}
Denote by $\mathcal{G}'$ the N\'eron $lft$-model of $G'$. By
definition, the length of the cokernel of the natural morphism
$$\Lie(h): \Lie(\mathcal{G}\times_R R') \to \Lie(\mathcal{G}')$$
from \eqref{eq-Lie} is equal to $c(G)d$. Writing $\Lie(h)$ in
Smith normal form, it is easily seen that the cokernel of
$$\det(\Lie(h))^{\vee}:  \Omega_{G'} \to \Omega_G\otimes_R R'$$
is isomorphic to $R'/(\varpi')^{c(G)d}$.
\end{proof}

\begin{prop}\label{prop-unram}
Let $R\to S$ be a flat local homomorphism of discrete valuation
rings of ramification index one (in the sense of
\cite[3.6.1]{neron}) and denote by $L$ the quotient field of $S$.
We denote by $\mathcal{G}^L$ the N\'eron $lft$-model of $G\times_K
L$.
\begin{enumerate}
\item The natural morphism
$$\mathcal{G}\times_R S\to \mathcal{G}^L$$
 is an isomorphism. In particular, it induces an isomorphism of
$S$-modules
$$\Omega_{G}\otimes_R S\cong \Omega_{G\times_K L}.$$
\item We have $c(G\times_K L)=c(G)$.
\end{enumerate}
This applies in particular to the case
 $S=\widehat{R^{sh}}$.
\end{prop}
\begin{proof}
(1)  The formation of N\'eron $lft$-models commutes with the base
change $R\to S$, by \cite[3.6.1]{neron}.

(2)
 This follows easily from (1).
\end{proof}

 The semi-abelian $K$-variety $G$ is bounded if and only if its
 N\'eron $lft$-model $\mathcal{G}$ is of finite type over $R$
\cite[10.2.1]{neron}. In that case, $\mathcal{G}$ is called the
N\'eron model of $G$.
 If $G$ is bounded, then for every gauge form $\omega$ on $G$, we
can consider the motivic integral
$$\int_{G}|\omega| \in \MRloc.$$
 In particular, we can consider the motivic integral of the ``motivic Haar measure'' $|\omega_G|$ associated to
 $G$. It does not depend on the choice of $\omega_G$, since
 $\omega_G$ is unique up to multiplication with a unit in $R$.

\begin{prop}\label{prop-vdim}
Let $G$ be a bounded semi-abelian $K$-variety of dimension $g$
 with N\'eron model $\mathcal{G}$.
 Let $\varpi$ be a uniformizer in $R$. Then for every integer
$\gamma$, we have
\begin{equation}\label{eq-haar}
\int_{G}|\varpi^{\gamma}\omega_G|=\LL^{-\gamma-g}[\mathcal{G}_k]
 \end{equation}
 in $\MRloc$. In particular, the virtual dimension of this motivic
 integral is equal to $-\gamma$.
 \end{prop}
\begin{proof}
 Since
 $\omega_G$ generates $\Omega_G$, we have
 $$\ord_C (\varpi^{\gamma}\omega_G)=\gamma$$ for every connected component $C$ of
 $\mathcal{G}_k$.
 Thus formula \eqref{eq-motint} becomes
\begin{eqnarray*}
\int_{G}|\varpi^{\gamma}\omega_G|&=& \LL^{-\gamma-g}\sum_{C\in
\pi_0(\mathcal{G}_k)}[C]
\\ &=&\LL^{-\gamma-g}[\mathcal{G}_k]
\end{eqnarray*}
in $\MRloc$, where the last equality follows from the scissor
relations in the Grothendieck ring.
\end{proof}

\subsection{Split subtori and bounded
varieties}\label{subsec-split} We've already mentioned in Section
\ref{subsec-mothaar} that a semi-abelian $K$-variety $G$ is
bounded if and only if the N\'eron $lft$-model $\mathcal{G}$ of
$G$ is quasi-compact. If $R$ is excellent (e.g., complete) and
 $k$ algebraically closed, then this is
also equivalent to the property that $G$ does not contain a split
torus \cite[10.2.1]{neron}. Since the boundedness condition plays
an important role in the definition of the motivic integral, we'll
now take a closer look at split subtori of semi-abelian varieties.
 The results in this section will allow us to establish
 an equivalence between Question \ref{ques-chai} and a Fubini property of motivic integrals (Theorem
 \ref{thm-fubini}).

Let $F$ be any field.  We denote by $(\mathrm{SpT}/F)$ the
category of split $F$-tori and by $(\SA/F)$ the category of
semi-abelian $F$-varieties
 (the morphisms in these categories are morphisms of algebraic $F$-groups).

 For every semi-abelian $F$-variety $G$, we
denote by $G_{\spl}$ the maximal split subtorus of $G$
\cite[3.6]{HaNi-compseries}. If $T$ is a split $F$-torus, then
every morphism of $F$-groups $T\to G$ factors through $G_{\spl}$,
by \cite[3.5]{HaNi-compseries}. Thus we can define a functor
$$(\cdot)_{\spl}:(\SA/F)\to (\mathrm{SpT}/F):G\mapsto G_{\spl}.$$

For every semi-abelian $F$-variety $G$, we put
$G^{\rb}=G/G_{\spl}$. Then $(G^{\rb})_{\spl}$ is trivial, by the
remark after \cite[3.6]{HaNi-compseries}.  It follows that every
morphism of semi-abelian $F$-varieties $f:G\to H$ induces a
 morphism of semi-abelian $F$-varieties
$$f^{\rb}:G^{\rb}\to H^{\rb},$$ so that we obtain
 a functor
$$(\cdot)^{\rb}:(\SA/F)\to (\SA/F):G\mapsto G^{\rb}.$$

\begin{lem}\label{lemm-maxssplit}
Let $F$ be a field, and let $f:G\to H$ be a smooth morphism of
semi-abelian $K$-varieties. Then the morphism
$f_{\spl}:G_{\spl}\to H_{\spl}$ is surjective.
\end{lem}
\begin{proof}
The identity component of $G\times_H H_{\spl}$ is a smooth and
connected closed subgroup of $G$, and thus a semi-abelian
$F$-variety \cite[5.2]{HaNi}. The morphism $$(G\times_H
H_{\spl})^o\to H_{\spl}$$ is still smooth. Therefore, we may
assume that $H$ is a split torus. It follows from
\cite[VI$_B$.1.2]{sga3.1} that the image of $f$ is closed in $H$,
and it is also open by flatness of $f$. Thus $f$ is surjective.

 We denote by $I$ the schematic image of the morphism $g:G_{\spl}\to H$.
 This is a closed subgroup of the split torus $H$. The quotient
 $H/I$ is again a split $F$-torus (it is
geometrically connected because $G_{\spl}$ is geometrically
connected, and it is smooth \cite[VI$_B$.9.2(xii)]{sga3.1} and
diagonalizable \cite[IX.8.1]{sga3.2}, so that it is a split
torus). The quotient $Q=G/G_{\spl}$ is an extension of an abelian
$F$-variety and an anisotropic $F$-torus, so that the morphism of
$F$-groups $Q\to H/I$ induced by $f$ is trivial. But this morphism
is surjective by surjectivity of $f$, so that $H/I$ must be
trivial, and $H=I$. Since the image of $G_{\spl}\to H$ is closed
\cite[VI$_B$.1.2]{sga3.1}, it follows that $G_{\spl}\to H$ is
surjective.
\end{proof}
\begin{prop}\label{prop-exact}
Let $F$ be  a field, and let
$$0\to G_1\to G_2\to G_3\to 0$$ be an exact sequence of
semi-abelian $F$-varieties.

\begin{enumerate}
\item The schematic image of $(G_1)^{\rb}\to (G_2)^{\rb}$ is a
semi-abelian subvariety $H$  of  $G_2$, and the morphism
$(G_1)^{\rb}\to H$ is an isogeny. Moreover, the sequence
\begin{equation}\label{eq-seqb1}
0\to H \to (G_2)^{\rb}\to (G_3)^{\rb}\to 0\end{equation} is exact.

\item If $(G_3)_{\spl}$ is trivial, then
\begin{equation}\label{eq-seq2}
0\to (G_1)^{\rb} \to (G_2)^{\rb}\to (G_3)^{\rb}\to 0\end{equation}
is exact.
\end{enumerate}
\end{prop}
\begin{proof}
Dividing $G_1$ and $G_2$ by $(G_1)_{\spl}$, we may assume that
$(G_1)_{\spl}$ is trivial (here we use that $G^{\rb}=(G/T)^{\rb}$
for every semi-abelian $F$-variety $G$ and every split subtorus
$T$ of $G$). Then $(G_1)^{\rb}=G_1$.

First, we prove (1). The kernel of the morphism $G_1\to
(G_2)^{\rb}$ is the closed subgroup
$\widetilde{G}_1=G_1\times_{G_2} (G_2)_{\spl}$ of $G_1$. It is
also a closed subgroup of $(G_2)_{\spl}$. By
\cite[VI$_B$.9.2(xii)]{sga3.1}, the quotient
$H=G_1/\widetilde{G}_1$ is smooth over $F$. Since
 $(G_2)_{\spl}$ is a split $F$-torus,
we know that $\widetilde{G}_1$ is a diagonalizable $F$-group
\cite[IX.8.1]{sga3.2}. Since $(G_1)_{\spl}$ is trivial, the
$F$-group $\widetilde{G}_1$ must be finite, so that  the
projection $G_1\to H$ is an isogeny. The morphism $G_1\to G_2$
induces a morphism of $F$-groups $H\to (G_2)^{\rb}$ that is a
closed immersion \cite[VI$_B$.1.4.2]{sga3.1}. It identifies $H$
with the schematic image of $G_1\to (G_2)^{\rb}$.
 It
follows from \cite[5.2]{HaNi} that $H$ is a semi-abelian
$F$-variety because it is a connected smooth closed subgroup of
the semi-abelian $F$-variety $(G_2)^{\rb}$.

It is clear that \eqref{eq-seqb1} is exact at the right, so that
it remains to prove that this sequence is also exact in the
middle. By the natural isomorphism
$$(G_2/G_1)/((G_2)_{\spl}/\widetilde{G}_1)\cong
(G_2/(G_2)_{\spl})/(G_1/\widetilde{G}_1)$$ it is enough to show
that $\widetilde{G}_2=(G_2)_{\spl}/\widetilde{G}_1$ is the maximal
split subtorus of $G_3=G_2/G_1$.
 But $(G_2)_{\spl}\to (G_3)_{\spl}$ is surjective by
Lemma \ref{lemm-maxssplit}, and its kernel is precisely
$\widetilde{G}_1$, so we see that $\widetilde{G}_2=(G_3)_{\spl}$.

Now we prove (2). Assume that $(G_3)_{\spl}$ is trivial. Then the
closed immersion $(G_2)_{\spl}\to G_2$ factors through $G_1$, and
since $(G_1)_{\spl}$ is trivial, we find that $(G_2)_{\spl}$ must
be trivial. Thus $G_i=(G_i)^{\rb}$ for $i=1,2,3$, and the result
is obvious.
\end{proof}

If $(G_3)_{\spl}$ is not trivial, it can happen that the sequence
$$0\to (G_1)^{\rb} \to (G_2)^{\rb}\to (G_3)^{\rb}\to 0$$
in Proposition \ref{prop-exact} is not left exact, as is shown by
the following example.

\begin{example}
Let $K$ be the field $\CC((t))$ of complex Laurent series and put
$K'=K((\sqrt{t}))$. The Galois group $\Gamma=\mathrm{Gal}(K'/K)$
is isomorphic to $\ZZ/2\ZZ$ and it is generated by the
automorphism $\sigma$ that maps $\sqrt{t}$ to $-\sqrt{t}$. Let
$G_2$ be the $K$-torus with splitting field $K'$ and character
module
$$X(G_2)=\ZZ\cdot e_1\oplus \ZZ\cdot e_2$$ where $\sigma$ permutes
$e_1$ and $e_2$.

Let $G_1$ be the maximal anisotropic subtorus of $G_2$. Its
character module is $X(G_1)=X(G_2)/X(G_2)^{\Gamma}$. We put
$G_3=G_2/G_1$. This is a split $K$-torus with character module
$X(G_3)=X(G_2)^{\Gamma}=\ZZ\cdot (e_1+e_2)$.

For every $K$-torus $T$ that splits over $K'$, we can consider the
trace map $$\mathrm{tr}_T:X(T)\to X(T)^{\Gamma}:x\mapsto
x+\sigma\cdot x.$$ It follows from the duality between tori and
their character modules that the maximal split subtorus of $T$ has
character module $X(T)/\mathrm{ker}(\mathrm{tr}_T)$ and that
$T^{\rb}$ is the $K$-torus with character module
$\mathrm{ker}(\mathrm{tr}_T)$. In this way, we see that
$(G_1)_{\spl}$ is trivial and that $(G_2)^{\rb}$ is the $K$-torus
with character module
$$\mathrm{ker}(\mathrm{tr}_{G_2})=\ZZ\cdot (e_1-e_2) .$$

Thus, applying the functor $(\cdot)^{\rb}$ to the exact sequence
of $K$-tori
$$0\to G_1 \to G_2\to G_3\to 0,$$
we obtain the sequence \begin{equation}\label{eq-notexact} 0\to
G_1 \to (G_2)^{\rb}\to 0\to 0\end{equation} and the morphism of
$K$-tori $G_1\to (G_2)^{\rb}$ corresponds to the morphism of
character modules
$$\alpha:\ZZ\cdot (e_1-e_2)\to X(G_2)/X(G_2)^{\Gamma}.$$
The morphism $\alpha$ is injective but not surjective; its
cokernel is $$X(G_2)/(\ZZ\cdot (e_1-e_2)+\ZZ\cdot (e_1+e_2))\cong
\ZZ/2\ZZ$$ with trivial $\Gamma$-action.  Therefore,
\eqref{eq-notexact} is not exact. More precisely, the morphism
$$G_1\to (G_2)^{\rb}$$ is an isogeny with kernel $\mu_{2,K}$.
\end{example}

\begin{prop}\label{prop-boundsab}
Assume that $R$ is excellent and that $k$ is algebraically closed.
For every semi-abelian $K$-variety $G$, the quotient $G^{\rb}$ is
a bounded semi-abelian $K$-variety.
\end{prop}
\begin{proof}
Since $(G^{\rb})_{\spl}$ is trivial, this follows immediately from
\cite[10.2.1]{neron}.
\end{proof}

\section{Chai's conjecture and Fubini properties of motivic
integrals}\label{sec3}
\subsection{Chai's conjecture and Haar measures} Let $$0\to T\to G\to A\to 0$$
 be a a short exact sequence of semi-abelian $K$-varieties (as the notation suggests, the main example we
 have in mind is the Chevalley decomposition of a semi-abelian $K$-variety $G$ as in Conjecture \ref{conj-chai}, but we will work in greater generality).
 The sequence of $K$-vector spaces
 $$0\to \Lie(T)\to \Lie(G)\to \Lie(A)\to 0$$ is exact, and by dualizing and taking determinants, we
 find a
canonical isomorphism of $K$-vector spaces
$$\Omega_{G}\otimes_R K\cong (\Omega_T\otimes_R \Omega_A)\otimes_R
K.$$ In this way, we can view $\Omega_T\otimes_R \Omega_A$ as an
$R$-lattice in $\Omega_G\otimes_R K$.

\if false
\begin{lem}\label{lemm-inc}
We have an inclusion
$$\Omega_T\otimes_R \Omega_A \subset \Omega_G$$ of $R$-lattices
inside $\Omega_G\otimes_R K$.
\end{lem}
\begin{proof}

\end{proof}
\fi

The following proposition is implicit in the proof on pages
724--725 of \cite{chai} (proof of Proposition 4.1 in {\em loc.
cit.} when the residue field is finite).
\begin{prop}\label{prop-add} Assume that $T$ is a torus. Let $\omega_T$ and
 $\omega_A$ be generators of $\Omega_T$, resp.
$\Omega_A$. Let $\varpi$ be a uniformizer of $R$, and denote by
$\gamma$ the unique integer such that
$\varpi^{-\gamma}(\omega_T\otimes \omega_A)$ generates the
$R$-module $\Omega_G$. Then $$c(G)=c(T)+c(A)+\gamma.$$ In
particular, $c(G)-c(T)-c(A)$ belongs to $\ZZ$.
\end{prop}
\begin{proof}
By Proposition \ref{prop-unram}, we may assume that $R$ is
complete and that $k$ is algebraically closed. Suppose that
$$\omega_G:=\varpi^{-\gamma}(\omega_T\otimes \omega_A)$$ generates $\Omega_G$.  Let
$K'$ be a finite separable extension of $K$ such that
$G'=G\times_K K'$ has semi-abelian reduction, and denote by $R'$
the normalization of $R$ in $K'$. Then $A'=A\times_K K'$ has
semi-abelian reduction and $T'=T\times_K K'$ is split
\cite[4.1]{HaNi-compseries}.

We denote by $\varpi'$ a uniformizer of $R'$, and by $d$ the
ramification degree of the extension $K'/K$. By Proposition
\ref{prop-bcdiff}, the $R'$-module $\Omega_{G'}$ is generated by
$$(\varpi')^{c(G)d } \omega_{G},$$
 and the analogous
property holds for $A$ and $T$. We denote by $\mathcal{G}'$,
$\mathcal{T}'$ and $\mathcal{A}'$ the N\'eron $lft$-models of
$G'$, $T'$ and $A'$, respectively. By the universal property of
the N\'eron $lft$-model, the exact sequence
$$0\to T'\to G'\to A'\to 0$$ extends uniquely to a sequence of
$R'$-group schemes
$$0\to \mathcal{T}'\to \mathcal{G}'\to \mathcal{A}'\to 0$$
and this sequence is exact by \cite[4.8(a)]{chai}.
 It follows that $$\Omega_{G'}=\Omega_{T'}\otimes_{R'} \Omega_{A'}\subset \Omega_{G'}\otimes_{R'}K'$$
  so that both $(\varpi')^{c(G)d } \omega_{G} $ and
$$(\varpi')^{(c(T)+c(A))d} (\omega_{T}\otimes \omega_A)=(\varpi')^{(c(T)+c(A))d}\varpi^{\gamma }\omega_G $$ are generators
of the free $R'$-module $\Omega_{G'}$. Thus,
 we find that $$(\varpi')^{(c(G)-c(T)-c(A))d}\varpi^{-\gamma } $$ is a
 unit in $R'$. This means that its $\varpi'$-adic valuation is zero,
 so that
$$c(G)=c(T)+c(A)+\gamma,$$
which concludes the proof.
\end{proof}
\begin{remark}
Let
$$0\to T\to G\to A \to 0$$ be an exact sequence of semi-abelian
$K$-varieties, and let $\omega_T$ and
 $\omega_A$ be generators of $\Omega_T$ and
$\Omega_A$, respectively. In \cite[\S8.1]{chai}, Chai considers
the following statement:

\vspace{5pt} ($*$) {\em  One has
 $c(G)=c(T)+c(A)$ if and only if $\omega_T\otimes \omega_A$
 generates $\Omega_G$.}
\vspace{5pt}

If $T$ is a torus, then this is a corollary of Proposition
 \ref{prop-add}. However, if $T$ is not a torus, it is not clear
 to us how statement ($*$) can be proven, although Chai \change{hints} that it \change{may be} implicit in the proof on pages
724--725 of \cite{chai}. If $G$, $T$ and $A$ have
 semi-abelian reduction, then $c(G)=c(T)=c(A)=0$ so that
 statement ($*$) contains the following special case:

\vspace{5pt}
 ($**$) {\em  If $G$, $T$ and $A$ have semi-abelian reduction, then $\omega_T\otimes \omega_A$
 generates $\Omega_G$.}
\vspace{5pt}

 Even this property does not seem obvious, because the sequence of
 identity components of N\'eron $lft$-models $$0\to \mathcal{T}^o\to \mathcal{G}^o\to
 \mathcal{A}^o\to 0$$ might not be exact; see \cite[7.5.8]{neron} for an example where $T$, $G$ and $A$ are abelian varieties
 with good reduction and $\mathcal{T}^o\to \mathcal{G}^o$ is not a monomorphism.
 If  statement ($**$) is true, then the proof of Proposition \ref{prop-add}
shows that Proposition
 \ref{prop-add}, and thus  statement ($*$), are valid without the assumption that $T$ is a
 torus.
\end{remark}

\begin{lem}\label{lemm-split}
 Let $G$ be a semi-abelian
$K$-variety, and let $T$ be a split subtorus of $G$. Then
$c(G)=c(G/T)$. In particular, $c(G)=c(G^{\rb})$.
\end{lem}
\begin{proof}
We set $H=G/T$. By \cite[4.8(a)]{chai}, the canonical sequence of
group schemes $$0\to \mathcal{T}\to \mathcal{G}\to  \mathcal{H}\to
0$$ is exact, so that $\Omega_T\otimes \Omega_H=\Omega_G$.  Now
the result
 follows from Proposition \ref{prop-add} and the fact that
 $c(T)=0$ because $T$ has semi-abelian reduction.
\end{proof}

\subsection{Main results}
We can now state our main results.

\begin{theorem}\label{thm-fubini}
Let
$$0\to T\to G\to A\to 0$$
be an exact sequence of semi-abelian $K$-varieties, with $T$ a
torus. We put \begin{eqnarray*}
\widetilde{G}&=&(G\times_K \widehat{K^{sh}})^{\rb}
\\ \widetilde{A}&=&(A\times_K
\widehat{K^{sh}})^{\rb}.
\end{eqnarray*}
and we denote by $\widetilde{T}$ the schematic image of the
morphism
$$(T\times_K \widehat{K^{sh}})^{\rb}\to \widetilde{G}.$$

Then $\widetilde{T}$ is a $\widehat{K^{sh}}$-subtorus of
$\widetilde{G}$,
$$0\to \widetilde{T}\to \widetilde{G}\to \widetilde{A}\to 0$$
 is an exact sequence of bounded semi-abelian $\widehat{K^{sh}}$-varieties,
 and
$$c(G)=c(T)+c(A)$$ if and only if
\begin{equation}\label{eq-fubini}
\vdim \int_{\widetilde{G}}|\omega_{\widetilde{T}}\otimes
\omega_{\widetilde{A}}|= \vdim
\int_{\widetilde{T}}|\omega_{\widetilde{T}}|+ \vdim
\int_{\widetilde{A}}|\omega_{\widetilde{A}}|=0.\end{equation}
\end{theorem}
\begin{proof}
By Proposition \ref{prop-unram}, we may assume that $K$ is
complete and $k$ algebraically closed, so that
$\widehat{K^{sh}}=K$. By Proposition \ref{prop-exact}, we know
that $T^{\rb}$ and $\widetilde{T}$ are isogenous $K$-tori, so that
$c(T^{\rb})=c(\widetilde{T})$ by \cite[11.3 and 12.1]{chai-yu}.
 Thus, by Proposition \ref{prop-exact} and Lemma
\ref{lemm-split}, we may assume that  $\widetilde{T}=T$,
$\widetilde{G}=G$ and $\widetilde{A}=A$.
 Then $T$, $G$ and $A$ are bounded, by Proposition
 \ref{prop-boundsab}.
 so that we can take motivic
 integrals of gauge forms on $T$, $G$ and $A$.

Let $\varpi$ be a uniformizer in $R$. It follows from Proposition
\ref{prop-vdim} that
 $$\vdim
\int_{T}|\omega_{T}|+ \vdim \int_{A}|\omega_{A}|=0$$ and that
$$\vdim \int_{G}|\omega_{T}\otimes
\omega_{A}|$$ is equal to the unique integer $\gamma$ such that
$$\varpi^{\gamma}(\omega_{T}\otimes
\omega_{A})$$ generates the $R$-module $\Omega_G$. By Proposition
\ref{prop-add}, we know that $\gamma=0$ if and only if
$c(G)=c(T)+c(A)$.
\end{proof}
\begin{remark}\label{rem-noiso}
 In Conjecture \ref{conj-chai}, $T$ is a torus and $A$ is an
 abelian variety. This implies that $A_{\spl}$ is trivial, so that
 $A=A^{\rb}$ and
$$0\to T^{\rb}\to G^{\rb}\to A\to 0$$ is exact by Proposition
\ref{prop-exact}. In this case, in the proof of Theorem
\ref{thm-fubini}, we do not need the fact that the base change
conductor of a torus is invariant under isogeny.
\end{remark}

\begin{theorem}\label{holds}
Assume that $K$ is complete and of characteristic zero and that
$k$ is algebraically closed. Let
$$0\to T\to G\to A\to 0$$
be an exact sequence of bounded semi-abelian $K$-varieties, with
$T$ a torus. Then
$$\vdim \int_{G}|\omega_{T}\otimes
\omega_{A}|= \vdim \int_{T}|\omega_{T}|+ \vdim
\int_{A}|\omega_{A}|=0.$$
\end{theorem}
We will prove Theorem \ref{holds} in Section \ref{subsec-proof},
using the model-theoretic approach to motivic integration in
\cite{mimix} and a new result on dimensions of motivic integrals
(Theorem \ref{badim}). As a corollary, we obtain a new proof of
the following theorem of Chai.
\begin{theorem}[Chai]\label{MT} Let
$$0\to T\to G\to A\to 0$$
be an exact sequence of semi-abelian $K$-varieties, with $T$ a
torus. Assume that $K$ is of characteristic zero. Then
$$c(G)=c(T)+c(A).$$
\end{theorem}
\begin{proof}This follows at once from Theorem \ref{thm-fubini}
and Theorem \ref{holds}.
\end{proof}

\begin{remark}\label{rem-fubini}
\change{
Theorem \ref{holds} is not a direct corollary of the Fubini theorem in
\cite{mimix} and the above results.
We need to combine the Fubini result \cite[12.5]{mimix} with the new result in
Theorem \ref{badim} and its corollary below, which compares dimensions of motivic parameter integrals under rather general conditions. By the lack of a definable section for the morphism $G\to A$ as in Theorem \ref{holds}, the motivic integral of $|\omega_G|$ over $G$ may not be equal to the product of the integrals of $|\omega_T|$ over $T$ and of $|\omega_A|$ over $A$. By the corollary to Theorem \ref{badim} and by the change of variables in \cite[12.4]{mimix}, this product survives at the rough level of virtual dimensions, which is sufficient to prove Theorem \ref{holds}. }
\end{remark}

\section{A comparison result for the dimensions of motivic integrals}

In this section we will work with a specific context falling under \cite{mimix} and define the dimension of motivic constructible functions at each point. These functions play an important role in motivic integration and in general Fubini results of \cite{mimix} and \cite{CL}. In order to control dimensions as desired for equation \eqref{eq-fubini} in Theorem \ref{thm-fubini}, we will compare the dimensions of the integrals of possibly different functions $F$ and $G$, when we are given that $F$ and $G$ have the same dimension in every point. Theorem \ref{badim} provides such a comparison result with parameters, and its corollary gives a similar comparison result for integrals on an algebraic variety with a volume form.


\subsection{Dimensions of motivic constructible functions}
We suppose in this section that $R$ is a complete discrete
valuation ring of characteristic zero with quotient field $K$ and
algebraically closed residue field $k$ of characteristic $p\geq
0$. We fix a uniformizer $\varpi$ in $K$. We will use some
terminology and results from \cite{mimix}, with precise references. Let $\cL_{\rm
high}(K)$ be the Denef-Pas language $\cL_{\rm high}$ as in \cite[\S2.3]{mimix} enriched with coefficients from $K$, and where the angular component maps $\ac_n$ for $n>0$ are given by $\ac_n(x)=x\varpi^{-\ord x}\bmod (\varpi)^n$ for nonzero $x\in K$. Let $\cT$ be the $\cL_{\rm
high}(K)$-theory of $K$.

Since $\cT$ falls under the combined Examples 1 and 4 of
\cite[\S3.1]{mimix}, we can use the theory of motivic integration of \cite{mimix}. Also, since $\cT$ is a complete theory, any definable subassignment $X$ is uniquely determined by the definable set $X(K)$, with notation from \cite[\S4.1]{mimix}. We will sometimes say definable set instead of definable subassignment.


We first define how to take (virtual) dimensions of several objects appearing in \cite{mimix}.
Write $R_n$ for the ring $R/(\varpi^n)$ and $\varpi_{n}$ for the
image of $\varpi$ in $R_n$. Let $\cL_{r}$ be the multisorted
language with sorts $R_n$ for integers $n>0$, on each $R_n$ the ring
language with coefficients from $R_n$, and with the natural
projection maps $p_{n,m}:R_n \to R_m$ for $n\geq m$. It follows from the quantifier elimination results of \cite{Pas} and \cite{Past} that any $\cL_{\rm high}(K)$-definable set $X\subset \prod_{i=1}^s R_{n_i}$ is already $\cL_r$-definable, see also \cite[Theorem 3.10]{mimix}.
To each $\cL_r$-definable set $X\subset \prod_{i=1}^s R_{n_i}$, we
associate an $\cL_{\rm ring}(k)$-definable set $\delta(X)$ as follows. If $R$
has mixed characteristic, the projection $p_{n,1}$ induces a
bijection from the set of $p^{n}$-th powers in $R_n$ to $k$, by
Hensel's Lemma, Newton's binomium, and the hypotheses on $K$. Let
us write $P_{p^n}$ for the set of $p^{n}$-th powers in $R_n$. Then
any $x$ in $R_n$ can be written uniquely as
$$
\sum_{i=0}^{n-1}  x_i \varpi_{n}^i,
$$
with $x_i\in P_{p^n}$, yielding a bijection $R_n\to k^n:x\mapsto (p_{n,1}(x_i))_i$ which is,
in fact, $\cL_{r}$-definable. If $R$ has equal characteristic
zero,
 we choose a retraction $k\to R$ of the ring morphism $R\to k$. This choice determines
 an isomorphism $R\to k[[\varpi]]$, and we identify $R_n\cong k[\varpi]/(\varpi^n)$ with the $k$-vector space $k^n$
 by means of the basis $1,\varpi,\ldots,\varpi^{n-1}$ of
$R_n$. In both cases, we obtain a bijection $\prod_{i=1}^s
R_{n_i}\to k^N$ with $N=\sum_{i=1}^s n_i$. This identification maps the $\cL_r$-definable subset $X$ of $\prod_{i=1}^s R_{n_i}$ onto an $\cL_{\rm ring}(k)$-definable subset of $k^{N}$
that we denote by $\delta(X)$.

Recall from \cite[\S7.1]{mimix} that $\cC_+({\rm Point})$ is the Grothendieck semi-ring of $\cL_{\rm
high}(K)$-definable subsets of Cartesian products of the form $\prod_{i=1}^s R_{n_i}$ up to definable isomorphisms, with scissor relations, with zero-element $[\emptyset]$, and localized with respect to $\LL$ and the elements $\LL^{i}-1$ for all $i>0$, where $\LL$ stands for the class of the affine line over $k$.
Clearly, $\delta$ induces a semi-ring morphism
$$
\cC_+({\rm Point}) \to \Mr,
$$
which we also denote by $\delta$. Recall that objects in $\Mr$ have a dimension by Definition \ref{def-dim}.


For an $\cL_{\rm
high}(K)$-definable set $Z$, $\cC_+(Z)$ is a relative variant of $\cC_+({\rm Point})$ over $Z$, see \cite[\S7.1]{mimix}. An object $\varphi\in \cC_+(Z)$ is called a motivic constructible function on $Z$. Moreover, for every $z\in Z(K)$, there is the evaluation map $i_z^*:\cC_+(Z) \to \cC_+({\rm Point})$ at $z$, and $i_z^*(\varphi)$ is called the evaluation of $\varphi$ at $z$. For an $\cL_{\rm high}(K)$-definable set $Z$, a point $z\in Z(K)$, and a function $\varphi$ in $\cC_+(Z)$, the dimension of $\varphi$ at $z$ is defined as $\dim (\delta( i_z^*(\varphi) ) )$ and is denoted by $\dim_z(\varphi)$. If $Z$ is the point and $\varphi\in\cC_+({\rm Point})$, we write $\dim(\varphi)$ instead of $\dim_{\rm Point}(\varphi)$.

\subsection{A comparison result}

In this section definable will mean for the language $\cL_{\rm
high}(K)$.
Recall that, for definable sets $X$, $Y$ and $Z\subset X\times Y$, under integrability conditions in the fibers of the projection $Z\to X$, called $X$-integrability, one can integrate $\varphi\in\cC_+(Z)$ in the fibers of the projection $Z\to X$ to obtain a function $\mu_{/X}(\varphi)$ in $\cC_+(X)$, see \cite[9.1]{mimix}. The method of \cite{CL} and \cite{mimix} for calculating integrals goes back to ideas by Denef in \cite{D84} in the $p$-adic case and to Pas \cite{Pas}, \cite{Past} in a pre-motivic setting.

Now we can state and prove our comparison result, stating that the
dimension of a motivic integral only depends on the dimensions of
the values of the integrand at each point.

\begin{theorem}\label{badim}
Let $F$ and $G$ be in $\cC_+(Z)$ and suppose that $Z\subset X\times Y$ for some definable sets $X$, $Y$ and $Z$. Suppose that $F$ and $G$ are $X$-integrable and
that
$$
\dim_z ( F  )  = \dim_z  ( G )\mbox{ for each point $z$ on $Z(K)$}.
$$
Then one has
$$
\dim_x ( \mu_{/X}(F)) = \dim_x ( \mu_{/X}(G) ) \mbox{ for each point $x$ on $X(K)$}.
$$
\end{theorem}
\begin{proof}
By projecting one variable at the time and by iterating the one
variable result, it suffices to consider the case where
two of the three values $n,m$, and $r$, are zero, and either $n=1$, $r=1$, or $m=1=(1)$. By the cell decomposition
theorem of \cite{Pas}, \cite{Past}, we may suppose that $n=0$. Indeed, via cell
decomposition, each integral over a valued field variable is
precisely calculated as a sum over $\ZZ$-variables and a
subsequent integral over residue ring variables, see \cite[\S8]{mimix}.

Recall that $F$
is a finite sum of terms of the form $a_i \otimes b_i$, with
$a_i\in \cP_+(Z)$ and $b_i\in \cQ_+(Z)$, and similarly for $G$, with notation from \cite[\S7.1]{mimix}. (The semi-ring $\cP_+(Z)$ is related to the value group, $\cQ_+(Z)$ to the residue field, and $\cC_+(Z)$ is a tensor product of both.)

If $n=r=0$, then we may suppose that the $a_i$ lie in $\cP_+(X)$, and
similarly for $G$, by Proposition 7.5 of \cite{mimix}. The result
of the theorem now follows from the definition in
\cite[\S6]{mimix} of $\mu_{/X}$ in this case and the following simple comparison property for dimensions
of constructible sets $A_i\subset \AA_k^{n_i}$. If, for certain
morphisms $f_i: \AA_k^{n_i}\to \AA_k^{n_3}$ for $i=1,2$, one has
that $f_1(A_1)=f_2(A_2)$ and for each $x\in \AA_k^{n_3}(k)$, the
dimension of $f_1^{-1}(x)\cap A_1$ equals the dimension of
$f_2^{-1}(x)\cap A_2$, then one has $\dim(A_1)=\dim(A_2)$.

Let us finally consider the case that $n=m=0$ and $r=1$. In this
case we may suppose that the $b_i$ lie in $\cQ_+(X)$, and
similarly for $G$, again by Proposition 7.5 of \cite{mimix}. In
the considered case, the theorem follows from the definition of
$\mu_{/X}$ of \cite{mimix} and the following two observations.
 For any $a\in \AA$, where $a=a(\LL)$ is thus a rational function in
$\LL$ of a specific kind, one has that $\vdim (a)$ equals the
degree of the rational function $a(\LL)$, where the degree of a
rational function is the degree of its numerator minus the degree
of its denominator, and where the degree of $0$ is defined as
$-\infty$. Secondly, there is the following elementary comparison property for the degrees of rational functions.
Consider, for each $i\in\ZZ$, an integer $n_i\geq 0$,
and a polynomial $a_i(x)$ over $\ZZ$ in one variable $x$ such that
$a_i(q)\geq 0$ for each real $q>1$. If there is a
rational function $r(x)$ such that $\sum_{i\in\ZZ} a_i(q)/q^{n_i}$
converges and equals $r(q)$ for each real $q>1$, then the
following equality holds
$$
\deg ( r(x) ) = \max_i \deg \big( { \frac{a_i(x)}{x^{n_i}} }
\big),
$$
where $\deg$ stands for the degree.
Since summation of non-negative functions over $\ZZ$ in \cite[\S5]{mimix} is calculated and defined by considering specific sums of rational functions in $\LL$ and by evaluating in real numbers $q>1$, the result follows.
\end{proof}

By working with affine charts over $K$ one may consider a variety $V$ over $K$ as a definable subassignments, and one defines $\cC_+(V)$ correspondingly, see \cite[\S12.3]{mimix}. Let us write  $\int^{\rm CL}$ for the integral as defined in \cite[\S12.3]{mimix}, to distinguish with the integrals from section \ref{subsec-motint} of this paper. The definition of these integrals in \cite[\S12.3]{mimix} is based on finite affine covers of $V$ over $K$, on finite additivity for motivic integrals, and on the change of variables formula.

\begin{cor}\label{corbadim}
Let $V$ be an algebraic variety over $K$ with a volume form $\omega_V$. Let $F$ and $G$ be integrable functions in $\cC_+(V)$ such that, for each $x\in V(K)$, one has $\dim_x ( F)  = \dim_x(G)$. Then their integrals have the same dimension:
$$
\dim (\int^{\rm CL}_V F|\omega_V| ) = \dim (\int^{\rm CL}_V G |\omega_V| ).
$$
\end{cor}
\begin{proof}
This follows immediately from Theorem \ref{badim} and the definition of the integrals $\int^{\rm CL}$ in \cite[\S12.3]{mimix}.
\end{proof}

\subsection{Gelfand-Leray residues}

 Let $f : X \to Y$ be a smooth morphism between smooth
equidimensional varieties over $K$. Let $m$ be the dimension of
$Y$ and let $m+n$ be the dimension of $X$. Let $\omega_X$ and
$\omega_Y$ be differential forms of maximal degree on $X$ and $Y$,
respectively. Assume that $\omega_Y$ is a gauge form, that is, a
generator of the line bundle $\Omega^m_{Y/K}$ at each point of
$Y$.

 Since $f$ is smooth, the fundamental sequence of locally free
 coherent $\mathcal{O}_X$-modules
 $$0\to f^*\Omega^1_{Y/K}\to \Omega^1_{X/K}\to \Omega^1_{X/Y}\to
 0$$ is exact \cite[17.2.3]{ega4.4}. Taking maximal exterior
 powers, we obtain an isomorphism
 $$f^*\Omega^m_{Y/K}\otimes \Omega^n_{X/Y}\to
 \Omega^{m+n}_{X/K}.$$ Locally, this isomorphism is defined by
 $$\phi\otimes \eta\mapsto \phi\wedge
 \widetilde{\eta}$$ where $\widetilde{\eta}$ is any lift of $\eta$
 to $\Omega^n_{X/K}$. Since $f^*\omega_Y$ generates the line
 bundle
 $f^*\Omega^m_{Y/K}$, we obtain an isomorphism
 $$\Omega^n_{X/Y}\to
 \Omega^{m+n}_{X/K}$$ that is locally defined by
 $$\eta\mapsto f^*\omega_Y\wedge \widetilde{\eta}.$$ The inverse
 image of $\omega_X$ under this isomorphism is called the
 Gelfand-Leray form associated to $\omega_X$ and $\omega_Y$, and  denoted by $\omega_X/\omega_Y$. It induces a differential form
 of maximal degree on each of the fibers of $f$.

\subsection{Proof of Theorem \ref{holds}}\label{subsec-proof}

Finally, we can prove Theorem \ref{holds}. It follows from
Proposition \ref{prop-vdim} that
$$\vdim \int_T |\omega_T|=\vdim \int_A |\omega_A|=0.$$
It is proven in Proposition 12.6 of \cite{mimix} that the theory
of motivic integration developed there can be used to compute the
motivic integrals defined by the formula \eqref{eq-motint}. In
particular, the dimensions of the respective motivic integrals are
the same, so that we can use the corollary of Theorem \ref{badim} and the results
in \cite{mimix} to prove Theorem \ref{holds}.

We denote the projection morphism $G\to A$ by $f$. Since $H^1 (K,
T)=0$ by \cite[4.3]{chai}, the map $f(K): G (K)\to A(K)$ is
surjective. For every $a\in A(K)$, we set $G_a=f^{-1}(a)$.  If we
choose a point $x$ in $G_a(K)$, then the multiplication by $x$
defines an isomorphism $\tau_x:T\to G_a$. Since the relative
differential form $(\omega_{{T}}\otimes \omega_{{A}} ) /
\omega_{A}$ on $G$ is invariant under translation, the pullback
through $\tau_x$ of its restriction to $G_a$ equals $\omega_T$.
Thus, by the change of variables formula in \cite[12.4]{mimix}, we
have
$$
\int^{\rm CL}_{G_a} | ( \omega_{{T}}\otimes \omega_{{A}} ) /
\omega_{{A}} | =  \int^{\rm CL}_T |\omega_T|
$$
 for each $a$ in $A(K)$.
Hence, by the Fubini property in
  \cite[12.5]{mimix}, we find that
$$ \int^{\rm CL}_{G} | \omega_{{T}}\otimes \omega_{{A}} | =  \int^{\rm CL}_{A} \psi |\omega_A|
$$
where $\psi$ is a motivic constructible function on $A$ such that $\dim_a(\psi)=0$ for each $a\in A(K)$. Now Theorem \ref{corbadim} with $V=A$
implies that
$$\vdim \int^{\rm CL}_{A} \psi |\omega_A|=\vdim \int^{\rm CL}_A|\omega_A|=0.$$
Combining the above equations, we find
$$
\dim \int_{G} | \omega_{{T}}\otimes \omega_{{A}} | =  \dim \int^{\rm CL}_{G} | \omega_{{T}}\otimes \omega_{{A}} | = 0,
$$
which concludes the proof of Theorem \ref{holds}.


\bibliographystyle{amsplain}

\end{document}